\documentclass[12pt, a4paper]{amsart}

\usepackage{amssymb,mathrsfs,enumitem}

\usepackage[colorlinks]{hyperref}
\usepackage{tikz}
\usepackage{tikz-cd}

\newtheorem{theorem}{\rm\bf Theorem}[section]
\newtheorem{proposition}[theorem]{\rm\bf Proposition}
\newtheorem{lemma}[theorem]{\rm\bf Lemma}

\newtheorem*{theorem*}{Theorem}
\newtheorem*{theorem 1}{\rm\bf Proposition 1}
\newtheorem*{theorem 2}{\rm\bf Proposition 2}
\newtheorem*{conj*}{Conjecture}

\theoremstyle{definition}
\newtheorem{definition}[theorem]{\rm\bf Definition}

\theoremstyle{remark}
\newtheorem{remark}[theorem]{\rm\bf Remark}

\theoremstyle{plain}
\newtoks\thehProclaim
\newtheorem*{Proclaim}{\the\thehProclaim}

\theoremstyle{definition}
\newtoks{\thehRemark}
\newtheorem*{Remark}{\the\thehRemark}

\renewcommand{\leq}{\leqslant}

\def\scal#1#2{\langle #1, #2\rangle}
\def\R#1{\mathbb{R}^{#1}}

\DeclareMathOperator{\re}{\mathrm{Re}}

\DeclareMathOperator{\Sym}{\mathrm{Sym}}

\DeclareMathOperator{\trace}{\mathrm{tr}}

\def\Idm{\mathrm{Idm}}

\def\F{\mathbb{A}}

\def\half#1#2{\begin{matrix}\frac{#1}{#2}\end{matrix}}

\begin{document}

\dedicatory{Dedicated to V.G. Maz'ya on the occasion of his 80th birthday}

\title[]{Spectral properties of nonassociative algebras and breaking regularity for nonlinear elliptic type PDEs}

\author{Vladimir G. Tkachev}

\address{Link\"oping University, Department of Mathematics\\
              SE-581 83,               Sweden}
\email{vladimir.tkatjev@liu.se}



\begin{abstract}
In this paper, we address the following question: Why certain nonassociative algebra structures emerge in the regularity theory of elliptic type PDEs and also in constructing nonclassical and singular solutions? The aim of the paper is twofold. Firstly, to give a survey of diverse examples on nonregular solutions to elliptic PDEs with emphasis on recent results on nonclassical solutions to fully nonlinear equations. Secondly, to define an appropriate algebraic formalism which makes the analytic part of the construction of nonclassical solutions more transparent.
\end{abstract}

\thanks{The author thanks the anonymous referee for several helpful comments.}

\date{12/NOV/2018}

\maketitle

\section{Introduction}

The first examples of  nonassociative algebras (including octonions and Lie algebras) appeared in the mid-19th century. Since then the theory has evolved into an independent branch of algebra, exhibiting many points of contact with other fields of mathematics, physics, mechanics, biology and other sciences. Algebras whose associativity is replaced by identities (as the famous Jacobi identity in Lie algebras or the Jordan  identity $[\mathrm{ad}(x),\mathrm{ad}(x^2)]=0$ in Jordan algebras) were a central topic in mathematics in the 20th century.

Very recently,  certain commutative nonassociative algebra structures emer\-ged in a very different context:  regularity of viscosity solutions to fully nonlinear elliptic PDEs. More precisely, it appears that subalgebras of simple rank three formally real Jordan algebras can be employed for constructing truly viscosity solutions to uniformly elliptic Hessian equations. To explain this connection, we first review some relevant concepts and results.

\subsection{Linear and quasilinear elliptic type PDEs}\label{sec:linear}
Starting with the pioneering works of S.N. Bernstein, the maximum principle and a priori estimates have provided decisive instruments in proving the existence of solutions of general linear and quasilinear elliptic type PDEs $Lw=0$. Important progress in this direction was made in the late 1950s by E.~De Giorgi, J. Moser and J.~Nash by establishing fundamental a priori regularity  results. A fundamental result of the regularity theory asserts that every solution $w$ from a `natural' class $W^{1,2}$ of $Lw=0$ must be H\"older continuous, i.e. $w\in C^{0,\alpha}$ for a certain $\alpha>0$.

On the other hand, the important counterexamples independently constructed in 1968 for scalar equations by V.G. Maz'ya \cite{Mazya68} and for linear elliptic systems by De Giorgi \cite{DeG68} and E.~Giusti and M. Miranda \cite{GiuMiranda68} showed that there is no natural generalization of De Giorgi's theorem on the H\"older continuity for generalized solutions to strongly elliptic differential equations of order greater than two.

In particular, in \cite{Mazya68} Maz'ya constructed several examples of quasilinear strongly elliptic equations of order $2l\ge 2$ with $C^\infty$-coefficients whose generalized solutions are not $C^\infty$ regular.  The strong ellipticity of an operator $Lu=\sum D^{\alpha}(a_{\alpha\beta}D^{\beta}u)$ means  that  there exists $c>0$ such that the  inequality $\int_{\R{n}}\sum a_{\alpha\beta}D^{\alpha}\phi D^{\beta}\phi\,dx\ge
c\int_{\R{n}}\sum (D^{\alpha}\phi)^2\,dx$ holds for any $\phi\in C^{\infty}(\R{n})$ with a compact support, where summation is over all  $|\alpha|=|\beta|=l$. Then Maz'ya notes that the equation
$$
\nu \Delta^2 w+\varkappa \Delta \left(\frac{x_ix_j}{|x|^2}w_{x_ix_j}\right)+\varkappa \Delta \left(\frac{x_ix_j}{|x|^2}\Delta w\right)_{x_ix_j}+\mu \left(\frac{x_ix_jx_kx_l}{|x|^4}w_{x_ix_j}\right)_{x_kx_l}=0
$$
is strongly elliptic for $\varkappa^2<\mu\nu$ and the radial symmetric function  $w=|x|^a$ with
$$
a=2-\frac{n}2+\sqrt{\frac{n^2}{4}-\frac{(n-1)(\varkappa n+\mu)}{\nu+2\varkappa+\mu}}
$$
is a weak solution of the above equation in $W^{2,2}(B_1)$, where $B_1$ is the open unit ball in $\R{n}$. For $\varkappa=n(n-2)$, $\mu=n^2$ and $\nu=(n-2)^2+\varepsilon$, with $\varepsilon<0,$  the above strong ellipticity condition is satisfied and the corresponding solution $w=|x|^{a}$ is unbounded in $B$ for all dimensions $n\ge 5$ and sufficiently small $\varepsilon$.

A very related problem is the regularity of $p$-harmonic functions. It is well-known that for $p>1$, $p\ne2$, a weak (in the distributional sense) $p$-harmonic function is normally in the H\"older class $C^{1,\alpha}$ \cite{Ural68}, but need not to be a H\"older continuous or even continuous in a closed domain with non-regular boundary, as it follows from examples constructed by Krol' and Maz'ya in \cite{KrolMaz72}. On the other hand, if $\mathrm{ess} \sup |Du(x)|>0$ holds locally in a domain then $u(x)$ is in fact a real analytic function in $E$ \cite{Lewis77}. It is interesting whether the  converse non-vanishing property holds true. This question naturally leads to homogeneous $p$-harmonic functions (also called quasi-radial solutions in \cite{Aronsson86}); see the paper of M. Akman, J. Lewis, A. Vogel \cite{Akman} in the present volume for  the modern status of the problem and further discussion. We only mention that in  certain dimensions $n\ge 4$ one can explicitly construct rational homogenous $p$-harmonic functions for some distinguished values $p$, $2<p<n$ based on isoparametric forms  and generalizing the angle solutions constructed by Krol' and Maz'ya, see \cite{Tk18c}.

Another prominent example of breaking regularity is the existence of globally minimal cones and entire solutions of the minimal surface equation in dimensions $n\ge 8$ \cite{BGG}. The higher-dimensional examples constructed by Bombieri, De Giorgio and Giusti  \cite{BGG}  make an essential use of the quadratic isoparametric form
\begin{equation}\label{quadcone}
u_2(x):=x_1^2+x_2^2+x_3^2+x_4^2-x_5^2-x_6^2-x_7^2-x_8^2
\end{equation}
see also \cite{SS} and \cite{Simon89} for the general isoparametric case. 
The function $u_2(x)$ coincides with the norm in the split octonion\footnote{The author is grateful to Seidon Alsaody for pointing out this.} algebra \cite{Baez}. It is not clear, whether this  is merely an coincidence or there is some natural explanation in the context of singular and entire solutions.

On the other hand, Lawson and Osserman \cite{LOss} constructed non-parametric minimal cones of high codimensions providing examples of Lipschitz but non-$C^1$ solutions to the minimal surface equations, thereby making sharp contrast to the regularity theorem for minimal graphs of codimension one. The Lawson-Osserman examples based on Hopf foliations  corresponding to the classical division algebras $\F_d$: $\F_2=\mathbb{C}$, $\F_4=\mathbb{H}$ and $\F_8=\mathbb{O}$. More explicitly, for $d=2,4,8$ let
\begin{equation}\label{lawsonoss}
w(x)=\frac12\sqrt{\frac{2d+1}{d-1}}\,\frac{\eta\left(x\right)}{|x|}:\R{2d}\to \R{d+1},
\end{equation}
where
\begin{equation}\label{Hopf}
\eta(x)=(|z_1|^2-|z_2|^2, 2z_1\bar z_2), \quad x=(z_1,z_2)\in \F_d^2\cong \R{2d}.
\end{equation}

\noindent
Then the map $w(x):\R{2d}\to \R{d+1}$ provides a nontrivial Lipschitz solution  to the Bernstein problem in codimension $d+1$.

We finally mention some further classes of semilinear and quasilinear elliptic PDEs with   similar phenomena, for example, the recent entire solutions of the Ginzburg-Landau system constructed by A. Farina in \cite{Farina} by using  isoparametric forms; see also section~3 in \cite{Mingione06} for a recent survey of further counterexamples as well as  regularity results.

\subsection{Fully nonlinear elliptic type PDEs}
In the general, fully nonlinear case,  the regularity and existence issues become more involved and require two principal ingredients: (i) the Harnack inequality for solutions of the 2nd order non-divergence elliptic equations with measurable coefficients established in 1979 by N.~Krylov and M.V.~Safonov \cite{KrSaf1}, \cite{KrSaf2}, and (ii) a suitable extension of the concept of a  generalized solution, the so-called \textit{viscosity solution}, an important tool developed by Crandall-Lions,  Evans, Jensen and Ishii. We refer to \cite{Guide} for exact definitions and self-contained exposition of the basic theory of viscosity solutions and briefly note that the latter  concept generalizes weak solutions for divergence type elliptic equations by utilizing the maximum principle.  Viscosity solutions play a central role in many  problems from global geoemtry and analysis, and  fit naturally into the contexts of optimal control, differential and stochastic differential games as well as mathematical finance.

An important model example with a far-ranging set of applications in pure and applied mathematics is the  Dirichlet problem for the uniformly elliptic equation of the form
\begin{equation}\label{Hess}
\begin{array}{rl}
F(D^2 w)&=0\,\,\, \text{in} \quad B_1\\
w&=\phi\,\,\, \text{on} \quad \partial B_1
\end{array}
\end{equation}
with a uniformly elliptic operator $F$. The latter means that there exist $0<\lambda\le \Lambda$ such that the inequalities hold
$$
\lambda\|N\|\le F(M+N)-F(M)\le \Lambda \|N\|
$$
whenever $M,N$ are symmetric matrices with $N$ nonnegative semi-definite.

For many geometrical applications, an important particular case of \eqref{Hess} is the \textit{Hessian equations}, i.e. when $F(X)$ depends only on the eigenvalues of the (symmetric) matrix $X$. The Hessian equations include the classical examples of Laplace's equation, curvature Weingarten equation,  Monge-Amp\`{e}re equation and special Lagrange equation, and, more recently, the calibrated geometries and Dirichlet duality theory \cite{HL}, \cite{ReeseLaw09}. According to the general theory \cite{Guide}, a  fully nonlinear elliptic equation \eqref{Hess} always has a \textit{unique } continuous viscosity solution for any continuous data $\phi$. The question whether a viscosity solution is classical, i.e. twice differentiable, turned out to be very challenging. For certain cases the regularity is known:
\begin{enumerate}[label=(\roman*)]
\item
If $n=2$ then  $w$ is classical ($C^{2,\alpha}$) solution (Nirenberg \cite{Nirenberg53})\label{Niren}

\item
If $n\ge 2$  and $\Lambda\le 1+\epsilon(n)$ then  $w$ is $C^{2,\alpha}$
(Cordes \cite{Cordes56})

\item
 $w\in C^{1,\alpha}(B_{1/2})$,  $\alpha=\alpha(\Lambda, n)$  (Trudinger \cite{Trud89}, Caffarelli \cite{Caff89})

\item
$F$ {convex (concave)}   $\Rightarrow$ $w\in C^{2,\alpha}(B_{1/2})$ (Krylov, Evans \cite{Evans82})

\item
$w$ is $C^{2,\alpha}(B_1\setminus \Sigma)$,  where the Hausdorff dimension $\dim_{H}\Sigma<n-\epsilon$, $\epsilon=\epsilon(\Lambda,n)$ (Armstrong-Silvestere-Smart, \cite{ArmSilSmart}).
\end{enumerate}

Further regularity results including the modulus of continuity of viscosity solutions can be found  in \cite{Teix14}, \cite{ImbertSil16}; see also \cite{DKM} for the VMO-regularity and gradient estimates for viscosity solutions to nonhomogeneous fully nonlinear elliptic equations $F(D^2u)=f$ with $f\in L^p$.

\subsection{Truly viscosity solutions}\label{sec:trely}

It follows from the above that in the general case and in dimensions $n\ge 3$ only $C^{1,\alpha}$-regular viscosity can be expected. Until very recently the very existence of nonclassical solutions for $n\ge 3$ was not known. A major breakthrough in this direction was achieved in a series of papers of Nadirashvili and Vl{\u{a}}du{\c{t}} \cite{NV07}, \cite{NV08}, \cite{NV11a} by establishing first the existence of nonclassical and then  also truly singular (H\"older continuous) viscosity solution  in any dimension $n\ge 12$. The method of construction makes an  essential use of an accumulation property of the spectrum of the Hessian of a certain cubic form $u_{12}$ in $\R{12}$. The latter is also well known as a \textit{triality form} and comes from the Hamilton quaternions $\mathbb{H}$ \cite{Baez}.

\begin{theorem}[\cite{NV07}, \cite{NV11a}]
Let
\begin{equation}\label{u12}
\text{$u_{12}(x)=\re (z_1z_2z_3)$, where $x=(z_1,z_2,z_3)\in \mathbb{H}^3\cong \R{12}$}
\end{equation}
and
\begin{equation}\label{wquater}
w_{12,\alpha}(x) = \frac{u_{12}(x)}{|x|^\alpha}.
\end{equation}
If $1\le \alpha<2$ then $w_{12,\alpha}$ is a viscosity solution in $\R{12}$ of a uniformly elliptic Hessian equation \eqref{Hess} with a smooth $F$.
\end{theorem}

In view of the appearance of the quaternions in \eqref{wquater}, it is natural to expect that the corresponding cubic form $u_{24}$ based on the octonions also produces singular viscosity solutions. The paper \cite{NV11b} establishes that this is actually the case. Furthermore, in \cite{NTV} a nonclassical viscosity solution $w_5$ in $\R{5}$ based on the Cartan isoparametric cubic form $u_5$ has been constructed and, shortly after, in \cite{NV13a} Nadirashvili and Vl{\u{a}}du{\c{t}} established that the corresponding  $w_{5,\alpha}$ is a singular viscosity solutions in $\R{5}$. More precisely, one has

\begin{theorem}[\cite{NTV}]\label{th12}
There a cubic form $u_{5}(x)$ such that
$w_{5,\alpha}(x)=\frac{u_{5}(x)}{|x|^\alpha}$ is a viscosity solution  in $B_{1}\subset\R{5}$. If $1\le \alpha<2$ then $w_{5,\alpha}$ is a viscosity solution in $\R{12}$ of a uniformly elliptic Hessian equation \eqref{Hess} with a smooth $F$. Furthermore, $w_5:=w_{5,1}$ satisfies \eqref{Hess} with
$$
F(D^2 w)=(\Delta w)^5+2^83^2(\Delta w)^3+2^{12}3^5\Delta w+2^{15}\det D^2(w).
$$
\end{theorem}

This establishes the lowest possible dimension known so far where singular viscosity solutions to uniformly elliptic equations may exist.

The algebraic part of the above construction  relies heavily on certain accumulation properties of the spectrum of the cubic forms $u_k$ involved. It was noticed in \cite{NV10a} that to obtain a nonclassical viscosity solution, the corresponding cubic form $u_n$ must be rather exceptional. Furthermore, numerical simulations show that a random (`generic') cubic form does \textit{not} produce a nonclassical solution to \eqref{Hess}. This makes it reasonable to ask why $u_{12}$, $u_{24}$ and $u_5$ are so exceptional and how to characterize the cubic forms producing singular/nonclassical viscosity solutions?

 We already mentioned that $u_{12}$ and $u_{24}$ come from the triality concept and are intimately related to exceptional Lie groups and the classical division algebras in dimensions $4$ and $8$ respectively. It follows from \cite{NV10a} that the corresponding triality cubic form $u_6$ over the field of complex numbers $\mathbb{C}$ also produces a viscosity solution but only for  a degenerated elliptic equation.

 The cubic form $u_5$ is also rather special and has many important connections in analysis and geometry. In the present context, at least the following properties are  essential:
\begin{enumerate}[label=(\alph*)]
\item \label{isoa} $u_5$ is the generic norm on (the trace free subspace of) rank 3 formally real Jordan algebra $\mathcal{H}_3(\R{})$ of symmetric $3\times 3$, and  matrices.
\item\label{isob}
$u_5$ is the simplest Cartan isoparametric cubic form.
\end{enumerate}
There exists a natural correspondence between   the concepts in \ref{isoa} and \ref{isob}, see  \cite{Tk14}. Explicitly, the cubic form $u_5$ is given by the determinant representation
\begin{equation}\label{u5}
u_5(x)= \left|
            \begin{array}{ccc}
              \half{1}{\sqrt{3}}x_{1}+x_2 & x_3 & x_4 \\
              x_2& \half{-2}{\sqrt{3}}x_{1} & x_5 \\
              x_4 & x_5 &  \half{1}{\sqrt{3}}x_{1}-x_2\\
            \end{array}
          \right|
\end{equation}

Taking into account the above observations, it is highly desirable to find a conceptual explanation of the existence  of nonclassical solutions and relate them to appropriate algebraic structures. In this paper, we discuss some possible approaches to these questions.

The appearance of Jordan algebras in \eqref{u5} and \ref{isoa} makes it natural to ask whether these algebras also relevant for the  examples in section~\ref{sec:linear}. As we already mentioned, the isoparametric form $u_2$ in \eqref{quadcone} reminisces the norm of split octonions. Moreover, the Hopf map \eqref{Hopf}  coincides with the multiplication in Clifford type Jordan algebras (the so-called spin-factors), see \cite{Baez}.

\begin{remark}
In all the above examples of non-regular solutions, including the $p$-harmonic homogenous examples, the regularity breaks at an isolated point where (the blow-down of) a solution has a H\"older type singularity. Remarkably, the asymptotic cone at the singular point is always a \textit{minimal cone}. It is also interesting to know whether  the appearance of the minimality in this context is mere coincidence or there is a general explanation of this phenomenon? This question also naturally arises in connection to phase transitions and the De Giorgi conjecture, see \cite{Savin10}.
\end{remark}

The rest of the paper is structured as follows. In section~\ref{sec:defin} we  discuss the main ideas underlying the construction of  homogeneous nonclassical viscosity solutions  of \eqref{Hess}. There we  also define an algebraic formalism which connects cubic forms with certain commutative nonassociative algebras. In this setting, the gradient of a cubic form determines the multiplicative algebra structure; in particular, algebra idempotents correspond to stationary points of the cubic form. A cubic form is called special if it generates a nonclassical solution. We reformulate this definition in terms of the Peirce spectrum. We discuss the Peirce decomposition and generic cubic forms in sections~\ref{sec:Jord} and \ref{sec:generic}. In section~\ref{sec:hess} we apply the algebraic approach to special forms. In the last section\ref{sec:isop} we revisit Cartan's isoparametric cubics and determine their Peirce spectrum.


\section{Ingredients and definitions}\label{sec:defin}

\subsection{Preliminaries}
Recall that according to the regularity result \cite{ArmSilSmart}, if $w$ is a nonclassical solution to \eqref{Hess} in the  unit ball $B_1$ then its singular set must be compactly contained in $B_1$. It is natural to consider with the simplest single point singularity and look for candidates for nonclassical viscosity solutions  in  the unit ball $B_1\subset \R{n}$ in the class of homogeneous solutions, i.e. the functions $w$ satisfying
$$
w(tx)=t^kw(x), \qquad x\in \R{n}, t>0.
$$
The ambient dimension must be $n\ge5$. Indeed,  for $n=2$ it follows from the Nirenberg result \ref{Niren} mentioned above that  \textit{any} (not necessarily homogeneous) viscosity solution is classical, and also if $3\le n\le 4$ then  it follows from \cite{NV13b} that a nonclassical viscosity solution, if exists, is never homogeneous (the case $n=3$ follows from a theorem of A.D.~Alexandroff in \cite{Alexandroff}).

Furthermore, it is natural to specify $w(x)$ and assume that it has the following  form:
\begin{equation}
\label{homogeneous}
w(x)=\frac{u(x)}{|x|^\alpha},
\end{equation}
 $u(x)$ being a cubic form\footnote{In fact, one can also construct some counterexamples $w$ as in \eqref{homogeneous} using isoparametric homogeneous forms $u$ of degree $3,4$ and $6$, see Chapter~5 in \cite{NTVbook}. Interestingly, there are also \textit{quadratic} isoparametric forms, but they are non-appropriate for constructing the nonclassical solutions.  In this paper we confine ourselves with the degree $3$ case.}  in the variable $x\in \R{n}$ and $\alpha\ge1$.
Under the made  assumptions, it follows from \cite{NTV} that the class of nonclassical viscosity solutions with $\alpha=1$ is non-empty for any dimension $n\ge5$.

In all known cases, the existence of a nonclassical viscosity solution \eqref{homogeneous} for a certain $u(x)$ with $\alpha=1$ guarantees that \eqref{homogeneous} with $1<\alpha<2$ will also be a \textit{singular} viscosity solution to a certain $F$ in \eqref{Hess}, see  Chapter~4 in \cite{NTVbook} (it is believed that this observation holds in general, though no proof is known). In view of this, we shall only consider  nonclassical solutions.

\subsection{Nonclassical viscosity solutions}
We briefly recall the main ideas of the construction following to \cite{NV10a}, \cite{NTV}. For more information and the proofs, the reader is referred to chapters~4 and 5 in \cite{NTVbook}. Let $w$ be an arbitrary homogeneous function of order 2 defined on $\R{n}$ and smooth in $\R{n}\setminus\{0\}$. Then the Hessian $D^2w(x)$ is homogeneous of order $0$ at any point $x\in \R{n}\setminus\{0\}$ and it induces the Hessian map
$$
D^2w:\R{n}\setminus \{0\} \to \Sym_n\;,
$$
where $\Sym_n$ denotes the space of symmetric $n\times n$ matrices with real coefficients. By the homogeneity assumption, we have
$$
D^2w(tx)=D^2w(x) \quad\text{for any $t>0$,}
$$
hence the Hessian $D^2w$ is completely determined by its values on the unit sphere $\partial B_1$. We denote by
$$
H_w=D^2w|_{\partial B_1}
$$
the corresponding restriction map. By $O(n)$ we denote the group of isometries in $\R{n}$.

Following to \cite{NV10a} we say that a symmetric matrix $A$ is $M$-\textit{hyperbolic} for some real $M\ge1$ if either all $\lambda_i=0$ or
\begin{equation}\label{hyperbolic}
\frac{1}{M} \le\frac{-\lambda_1}{\lambda_n}\le M,
\end{equation}
where $\lambda_1 \le \cdots \le \lambda_n$ denote the eigenvalues of $A$.

A subset $\mathcal{A}\subset \Sym_n$ is called {hyperbolic} if there exists a constant $M$ such that for any two matrices $A,B \in \mathcal{A}$, $A-B$ is $M$-hyperbolic.

The importance of the latter concepts is clear from the following  result.

\begin{lemma}[Main Lemma in \cite{NV07}]
\label{lem411}
Let function $w$ be defined by \eqref{homogeneous} with $\alpha=1$. Suppose that the restriction map $H_w:\partial B_1 \to  \Sym_n$ is a smooth embedding. Then
\begin{enumerate}[label=(\roman*)]
\item\label{itemhess1}
if $H_w(\partial B_1)$ is hyperbolic then $w$ is a viscosity solution of a uniformly elliptic equation \eqref{Hess} in $\R{n}$;
\item\label{itemhess2}
if additionally the (larger) subset $\{UAU^{-1}: \,A\in H_w(\partial B_1), U\in O(n)\}$ is  hyperbolic then $w$ is a solution of a Hessian uniformly elliptic equation \eqref{Hess}  in $\R{n}$.
\end{enumerate}
\end{lemma}

Thus, the above hyperbolicity condition is crucial for constructing viscosity solutions.  Note also that all examples discussed in section~\ref{sec:trely} satisfy in fact the stronger condition \ref{itemhess2}. This suggests the following definition.

\begin{definition}
A cubic form $u(x)$ is called \textit{special} if $w=u(x)/|x|$ satisfies hypotheses \ref{itemhess2} in Lemma~\ref{lem411}.
\end{definition}

It is also interesting to characterize the cubic forms $u$ which  satisfy the weaker condition \ref{itemhess1}. For continuity reasons, if $u$ satisfies \ref{itemhess1} then all cubic forms `sufficiently close' to $u$ will also satisfy \ref{itemhess1}.

\begin{remark}\label{rem23}
An example of such a form is the determinant $u_9(x)=\det X$, where $X\in \R{3\times 3}\cong \R{9}$ ($X$ is an arbitrary nonsymmetric matrix over $R{}$). Then it is easy to see that $u_9$ is the specialization of the triality form $u_{12}$ on the imaginary subspace $(\mathrm{Im}\,\mathbb{H})^3$, wee \eqref{u12}. On the other hand,  it is known, see Remark 4.1.5 in \cite{NTVbook}, that $u_9$ satisfies the ellipticity criterion \ref{itemhess1} and thus produces a nonclassical solution $w=u/|x|$ of a uniformly
elliptic fully nonlinear equation in $\R{9}$. However, a more sophisticated analysis reveals that this function fails to satisfy the  sufficient conditions \ref{itemhess2}  and thus may be not a solution of a Hessian uniformly elliptic fully nonlinear equation. Remarkably, $u_9$ is one of the exceptional Hsiang cubic forms emerging in classification of cubic minimal cones\cite[Chapter 6]{NTVbook}, cf. also with the remarks in Section~\ref{sec:last} below.
\end{remark}


%
%
%

\subsection{Algebras of cubic forms}\label{sec:algebrascubic}


The main obstacle for identifying which cubic forms $u$ are special is an a priori hard problem to control the spectrum of the difference $H_w(x)-H_w(y)$ for all possible pairs $x,y\in \partial B_1$. This problem becomes more tractable and conceptually more transparent if one pass to a certain natural nonassociative algebra attached to the cubic form $u$. Below, we recall some relevant concepts and results following \cite{Tk15b}, \cite{Tk18a}, \cite{Tk18b}; see also section~6.3 in \cite{NTVbook}.

By an algebra we always mean a (finite dimensional) vector space $V$ over the real numbers $\R{}$  with a multiplication on $V$, i.e. a bilinear map (denoted by juxtaposition) $(x,y)\to xy\in V.$  By an abuse of notation, we  denote by $V$ both the vector space and the corresponding algebra.

In this paper, by an algebra we  mean a commutative algebra.
A symmetric bilinear form  $b:V\times V\to \R{}$ on an algebra $V$ is called \textit{associating}  if
\begin{equation}\label{Qass}
b(xy,z)=b(x,yz), \qquad \forall x,y,z\in V.
\end{equation}
An algebra carrying  an associating bilinear form is called \textit{metrised}. The standard example is the generic trace form on a Jordan algebra \cite{McCrbook}.

The most striking corollary of the existence of an associating bilinear form is that the operator of left ($=$right) multiplication
$$
L_x:y\to xy=yx
$$
is  self-adjoint, i.e.
\begin{equation}\label{selff}
\scal{L_xy}{z}=\scal{y}{L_xyz}, \qquad \forall x,y,z\in V.
\end{equation}

Now we want to associate with an arbitrary cubic form  an algebra. Recall that a cubic form is a function $u:V\to \R{}$ such that its full linearization
$$
u(x,y,z)=u(x+y+z)-u(x+y)-u(x+z)-u(y+z)+u(x)+u(y)+u(z)
$$
is a trilinear form. The trilinear form  $u(x,y,z)$ is obviously symmetric and the cubic form is recovered by
\begin{equation}\label{16}
6u(x)=u(x,x,x).
\end{equation}

A positive definite symmetric bilinear form $\scal{x}{y}$ on $V$ is called an inner product. Given an inner product and a cubic form $u$, its \textit{gradient} $\nabla u(x)$ at $x$ is uniquely determined by the duality
\begin{equation}\label{ggrad}
\scal{\nabla u(x)}{y}=\frac{1}{2}u(x,x,y), \quad \forall y\in V.
\end{equation}
It follows from the above definitions, that thus defined,  $\nabla u$ coincides with the standard gradient of a function $u$.


\begin{definition}
Let $u$ be a cubic form on an inner product vector space $(V,\scal{}{})$. Define the multiplication $(x,y)\to xy$ as the unique element satisfying
\begin{equation}\label{muu}
\scal{xy}{z}=u(x,y,z) \text{ for all } z\in V
\end{equation}
The multiplication is commutative but maybe nonassociative. We call thus defined algebra on $V$ the \textit{algebra of the cubic form} $u$ and  denote it by $V(u)$.
\end{definition}

It is easy to see that $V(u)$ is a zero algebra (i.e. $V(u)V(u)=0$) if and only if the cubic form $u$ is identically zero.

An important corollary of the above definition and the symmetricity of the trilinear form $u(x,y,z)$ is that the inner product on the algebra $V(u)$ is associating, i.e. $V(u)$ \textit{is a metrised commutative algebra}.

Furthermore, it also follows from the definition  that given an inner product vector space $(V,\scal{}{})$, there is a canonical bijection between the vector space $\mathcal{C}(V)$ of all cubic forms on $V$  and commutative metrised algebra structures $\mathcal{A}(V)$ on $V$, where the correspondence $\mathcal{C}(V)\to \mathcal{A}(V)$ is given by \eqref{muu}, and the converse correspondence  $\mathcal{A}(V)\to \mathcal{C}(V)$ is  defined by
\begin{equation}\label{recover}
u(x):=\frac16 \scal{x^2}{x}.
\end{equation}

With the above concepts in hand, we  are able to identify  the standard calculus operations on $V$ as appropriate algebraic concepts on $V(u)$. First note that \eqref{ggrad} combined with \eqref{muu} immediately yields that the gradient of $u(x)$ is essentially the square of the element $x$ in $V(u)$:
\begin{equation}\label{grad}
\nabla u(x)=\frac{1}{2}xx=\frac{1}{2}x^2,
\end{equation}
and, similarly, the multiplication in $V(u)$ is explicitly recovered by
\begin{equation}\label{hess}
(D^2u(x))y=xy,
\end{equation}
where $D^2u(x)$ is the Hessian map of $u$ at $x$. This implies that the Hessian map of $u$ at $x$ is nothing else than the multiplication operator $L_x$:
\begin{equation}\label{hessL}
D^2u(x)=L_x.
\end{equation}
Then it follows from \eqref{Qass} \textit{that $L_x$ is a self-adjoint operator} with respect to the inner product $\scal{}{}$.

In summary, starting with a cubic form on an inner product vector space one can construct an algebra structure which translates the standard calculus into appropriate algebraic concepts.  Then the Peirce decomposition relative to an idempotent is a significant tool in identifying of the underlying nonassociative algebra structure, see \cite{Schafer}, \cite{McCrbook}. Therefore it is important to know whether the set of idempotents in $V(u)$ is empty or not. The next proposition answers this question in positive; its proof  can be found in our papers \cite{Tk18a}, \cite{Tk18b}.

\begin{proposition}\label{pro:main}
Let $V(u)$ be the algebra of a cubic form $u\not\equiv 0$ on an inner product vector space $V$. Denote by $E$ the (nonempty) set of stationary points of the variational problem
\begin{equation}\label{variational}
\scal{x}{x^2}\to \max\quad \text{ subject to a constraint}\quad \scal{x}{x}=1.
\end{equation}
Then for any $x\in E$, either $x^2=0$ or
$c:=\frac{x}{\scal{x^2}{x}}$ is an idempotent in $V$.
In particular, the set of idempotents of $V(u)$ is nonempty. Furthermore, if $E_0\subset E$ denote the subset of local maxima in \eqref{variational} then the corresponding idempotent  $c=x/\scal{x^2}{x}$ satisfies the extremal property
\begin{equation}\label{extremal}
\mathrm{spectrum}(L_c|_{c^\bot})\subset (-\infty, \half12],
\end{equation}
where $c^\bot=\{x\in V:\scal{x}{c}=0\}$. In particular, the eigenvalue $1$ of $L_c $ is simple (i.e. $c$ is a primitive idempotent).
\end{proposition}

\subsection{Jordan algebras}\label{sec:Jord}
Jordan algebras is an important class in the context of cubic forms and having numerous applications in mathematics and mathematical physics. Recall that a commutative algebra $V$ is called \textit{Jordan} if
\begin{equation}\label{Jordandef}
[L_x,L_{x^2}]=0.
\end{equation}
It is well known that \eqref{Jordandef} implies that any Jordan algebra is power-associative, i.e. the powers $x^n$ do not depends on associations \cite{JacobsonBook}.

In the most interesting for applications case when $V$ is formally real, the bilinear form $b(x,y):=\trace L_{xy}$ (the generic trace form) is positive definite  and associating, i.e. satisfies \eqref{Qass}. In their famous work \cite{JordanNeumann} Jordan, von Neumann and Wigner proved that the only simple formally real Jordan algebras are the Jordan algebras of $n\times n$ self-adjoint matrices over $\F_d$, $d=1,2,3$, the exceptional $27$-dimensional Albert algebra of $3\times 3$ self-adjoint  matrices over octonions, and the so-called spin-factors, or Clifford type Jordan algebras.

Recall that the (Peirce) \textit{spectrum} of an idempotent $c$ in a commutative algebra is called the spectrum of the corresponding  multiplication operator $L_c$. The algebra spectrum is the union of all idempotent spectra.

An important structure  property of any Jordan algebra $V$ is that the spectrum of \textit{any} idempotent is a subset of $\{0, \frac12, 1\}$. In particular, $V$ admits the so-called \textit{Peirce decomposition}:
\begin{equation}\label{Peirce}
V=V_{0}(c)\oplus V_{\frac12}(c)\oplus V_1(c),
\end{equation}
where $V_\lambda$ is the $\lambda$-subspace of $L_c$. Then $V_c(\lambda)$ are subalgebras of $V$ for $\lambda \in\{0,1\}$.

The eigenvalue $\frac12$ and the corresponding subspace is distinguished in many ways. For example, a Jordan algebra is simple if and only its $\frac12$-eigenspace is nontrivial for any idempotent. Moreover, $V_{\frac12}$ satisfies the Jordan fusion laws
\begin{equation}\label{halffusion}
V_c(\lambda)\,V_c(\half12)\subset V_c(\half12) \quad \forall \lambda \in\{0,1\}, \qquad V_c(\half12)\,V_c(\half12)\subset V_c(0)\oplus V_c(1).
\end{equation}
As we shall see, $\frac12$  plays an essential  role in constructing nonclassical solutions.

\subsection{Generic cubic forms}\label{sec:generic}

As we remarked in section~\ref{sec:trely} above, some numerical simulations support a believe that a randomly chosen cubic form is not special. To give the latter observation a  rigorous meaning we need to formalize what we mean by a generic cubic form. One natural way to do this is to combine the correspondence given in the proceeding section with the concept of a generic algebra introduced  recently in \cite{KrTk18a}.

Recall that   an algebra over $\mathbb{C}$ is called \textit{generic} if it has exactly $2^n$ idempotents (counting $x=0$). The idea comes back to the classical work of Segre  \cite{Segre}: given a fixed basis  of an arbitrary nonassociative algebra $V$ of dimension $n\ge 2$ over complex numbers $\mathbb{C}$, one can identify the multiplication with a degree 2 homogeneous map on $V$. In this setting, the idempotent defining relation $x^2=x$ becomes a system of quadratic polynomial equations on $\mathbb{C}^{n}$. It is well known that a generic (in the Zariski sense) polynomial system has the B\'{e}zout number of solutions equal to the product of the principal degrees of the system equations.  In our case, the B\'{e}zout number is $2^n$.

More precisely, let $\{e_i\}_{1\le i\le n}$ be an arbitrary fixed basis of a finite dimensional vector space $V$. Then any algebra structure on $V$ is uniquely determined  by its multiplication table $M:=(a_{ijk})_{1\le i,j,k\le n}$, where  $e_ie_j=\sum_{k=1}^n a_{ijk}e_k$. In this notation, the algebra product $f(x,y):=xy$ in $V$ is given component-wise by
\begin{equation}\label{product}
f_k(x,y):=\sum_{k=1}^n a_{ijk}x_iy_j,\quad k=1,2,\ldots, n.
\end{equation}
The idempotent defining relation becomes in this notation
\begin{equation}\label{systemQ}
f(x,x)=x.
\end{equation}
Then the algebra $V$ is generic if \eqref{systemQ} has exactly $2^n$ distinct solutions over $\mathbb{C}$. In particular, the latter implies that each solution is a non-degenerate point of  \eqref{systemQ}. The latter condition is important because it can be reformulated in terms of the spectrum of the algebra $V$.

Now it is natural to call a cubic form $u$ on $V$ generic if the corresponding algebra $V(u)$ is so.

Note that the above definitions work equally well in both analytic and algebraic setting. However, it is more preferable to work with the latter because the genericity concept is easily translated to the well-developed nonassociative algebra theory, including the Peirce decomposition.

 In \cite{KrTk18a},  the following criterium has been established.

\begin{proposition}
\label{pro:generic}
A commutative algebra $V$ over $\mathbb{C}$ is generic if and only if the spectrum of any idempotent in $V$ does not contain $\half12$.
\end{proposition}

Because the essential role  $\half12$ playing in the present context and also for the reader convenience, we give a sketch of the proof.

\begin{proof}
Note that  $\tilde{f}(x):=f(x,x)$ is a homogeneous of degree 2 endomorphism of $V$. Since  $V$ is commutative we have
$$
\tilde{f}(x+y)-\tilde{f}(x)-\tilde{f}(y)=f(x,y)+f(y,x)=2f(x,y).
$$
In particular, $\frac{\partial \tilde{f}_k}{\partial x_j}(x)=2f_k(x,e_j)$, which yields
$$
D\tilde{f}(x)=2L_x,
$$
where $D\tilde{f}(x)$ is the Jacobi matrix of $g$ at $x$. Rewriting the system \eqref{systemQ} as
$g_k(x):=\tilde{f}_k(x)-x_k=0$, $k=1,2,\ldots,n,$
we see that
$$
Dg(x)=2L_x-I.
$$
Thus, $Dg(x)$ is non-degenerate if and only if $2L_x-I$ is so, i.e. $\half12 $ is not in the spectrum of $L_x$. In order to finish the proof it remains to note that the system \eqref{systemQ} has the maximal finite number of solution if and only if each solution is a nondegenerate point of $g$, see for instance \cite{Cut}.
\end{proof}

We have already seen in section~\ref{sec:Jord} that any Jordan algebra has $\frac12$ in its spectrum and plays a distinguished role in the  classification of formally real algebras by Jordan, von  Neumann and Wigner \cite{JordanNeumann}.
In general, the exceptionality of the Peirce eigenvalue $\lambda=\frac12$ in the spectrum of many well established nonassociative algebra structures (including Jordan algebras, Bernstein algebras and general genetic algebras) is a rather common phenomenon. We  mention its very recent appearance in the context of axial algebras (generalizing the Monster algebra of the largest sporadic finite simple group) \cite{HSS18}).   Hsiang algebras emerging in the context of cubic minimal cones is another example of commutative nonassociative algebras where the presence of $\frac12$ plays a distinguished role, see chapter~6 in \cite{NTVbook}. In a wider context of algebras with identities, the universality of $\frac12$  has been recently discussed in \cite{Tk18e}.

Our next step is to explain why the Peirce eigenvalue $\frac12$ is also distinguished in the context  nonclassical solutions.

%
%

\subsection{The spectral properties of the Hessian of $w$}
\label{sec:hess}
It has been remarked in Lemma~3.2 in  \cite{NV10a} that the hyperbolicity condition (Lemma~\ref{lem411} above) for a cubic form $u$ would be fulfilled  if there would exist $\delta>0$ such that for any fixed direction $d\in V$ the spectrum $\lambda_1(d)\le\ldots\le \lambda_n(d)$ of the quadratic form $\partial_d u(x)$ satisfies
\begin{equation}\label{delta}
\rho(d):=\max\{\frac{\lambda_1(d)}{\lambda_3(d)},\,
\frac{\lambda_n(d)}{\lambda_{n-2}(d)}\}<2-\delta.
\end{equation}
Unfortunately (see the remark after Lemma~3.2 in  \cite{NV10a}), the above condition is too strong and is failed for certain $d$ for \textit{any} cubic form. Indeed, using  \eqref{grad}, the quadratic form $\partial_d u(x)$ may be rewritten as
$$
\partial_d u(x)=\half12\scal{x^2}{d}=\half12\scal{x}{xd}=
\half12\scal{x}{L_dx},
$$
therefore its spectrum coincides with the spectrum of the scaled multiplication operator $\frac12L_d$. Now, choosing  $d\in E_0$ in the notation of  Proposition~\ref{pro:main}, it follows from \eqref{extremal} and the fact that $d$ is an idempotent that $\lambda_1(d)=1$ and $\lambda_2(d)\le \frac12$, hence $\rho(d)\ge2$.

Thus, the failure if the gap inequality \eqref{delta} is directly related to the Peirce spectrum of certain idempotents\footnote{These idempotents have interesting extremal properties and can be related to Clifford type Jordan algebras, see our recent paper \cite{Tk18b}} in the corresponding algebra $V(u)$. Below we try to clarify its nature in more detail and relate to generic cubic forms.

Let $u\not \equiv 0$ be an arbitrary cubic form on an inner product vector space $V$. Using \eqref{recover} we rewrite \eqref{homogeneous} for $\alpha=1$ as
$$
w(x)=\frac{6u(x)}{|x|}=\frac{\scal{x^2}{x}}{|x|}.
$$
Then we have for the directional derivative
\begin{equation}\label{D1f}
\scal{\nabla w(x)}{y}= \partial_y w|_{x}=\frac{3\scal{x^2}{y}|x|^2-\scal{x^2}{x}\scal{x}{y}}{|x|^3},
\end{equation}
implying by the duality the expression for the gradient
$$
\nabla w(x)=\frac{3x^2|x|^2-\scal{x^2}{x}x}{|x|^3}.
$$
Arguing similarly, we find for the Hessian
\begin{equation}\label{D2f}
H(x):=D^2w(x)=
\frac{6}{|x|}L_x-\frac{\scal{x^2}{x}}{|x|^3}-3\frac{ x\widehat{\otimes} x^2}{|x|^3}+3\frac{x\otimes x}{|x|^5}\scal{x^2}{x},
\end{equation}
where
$$
(a\otimes b)(x)=a\scal{b}{x}, \qquad
a\widehat{\otimes} b=a\otimes b+b\otimes a.
$$
Note that the Hessian is homogeneous degree 0 but it is an \textit{odd} map. i.e. $H(-x)=-H(x)$. Furthermore,
\begin{equation}\label{tracew1}
\Delta w(x)=\trace D^2w(x)=
\frac{6}{|x|}\trace L_x-(n+3)\frac{\scal{x^2}{x}}{|x|^3}.
\end{equation}

%
%
%

Recall that $u$ is special if there exists $M\ge1$ such that $H(x)-H(y)$ is $M$-hyperbolic for any pair $x,y\in V$. The following elementary observation immediately follows if one chooses $y=-x$ in and the oddness of $H(x)$.

\begin{lemma}
If $w$ is special then the set $\{H(x)\}_{|x|=1}$ is hyperbolic.
\end{lemma}

The particular case when $x=c$ is an idempotent of $V(u)$ is very special because $x^2$ and $x$ coincide, implying considerable simplifications in \eqref{D2f}, such that  the spectrum of $H(c)$ can be calculated explicitly.

\begin{lemma}
\label{lem:spec}
If  $c$ is a nonzero idempotent of $V(u)$ then
\begin{equation}\label{spec1}
\mathrm{spectrum}(H(c))=\left\{\frac{2}{|c|},\frac{6\lambda_1-1}{|c|} ,\ldots, \frac{6\lambda_{n-1}-1}{|c|} \right\}
\end{equation}
where $1,\lambda_1,\ldots,\lambda_{n-1}$ are the eigenvalues of $L_c$ counting the multiplicities. In particular, the characteristic polynomial of $H(c)$ is given by
\begin{equation}\label{charH}
\chi_{H(c)}(t)=\frac{6^n(6|c|t-2)}{|c|^n(|c|t-5)}
\cdot \chi_c\left(\frac{1+|c|t}{6}\right)
\end{equation}
where $\chi_c(z)$ is the characteristic polynomial of $L_c$.
\end{lemma}

\begin{proof}
Indeed, since $c^2=c$  we obtain from \eqref{D2f}
\begin{equation}\label{idem1}
H(c)=
\frac{1}{|c|}\left(6 L_c-1-\frac{3}{|c|^2}\,c\otimes c\right).
\end{equation}
In particular, $H(c)c=\frac2{|c|} c$, i.e. $c$ is an eigenvector of $H(c)$ with eigenvalue $\frac2{|c|}$. Since $H(c)$ is self-adjoint, the orthogonal complement $c^\bot$ is its invariant subspace. We have
$$
H(c)=
\frac{1}{|c|}(6 L_c-1)\quad \text{ on }c^\bot.
$$
This yields \eqref{spec1}, and therefore \eqref{charH}.
\end{proof}

\begin{remark}
As an immediate corollary of \eqref{charH} we obtain
\begin{equation}\label{charHc}
\chi_{c}(z)=(z-1)\frac{|c|^n}{6^n}\cdot \frac{\chi_{H(c)}\left(\frac{6z-1}{|c|}\right)}{z-\frac12}.
\end{equation}
We point out a remarkable  appearance of the eigenvalue $\frac12$ in the denominator. This implies  by virtue of Proposition~\ref{pro:generic} that if $u$  is a generic cubic form then the eigenvalue $\lambda=\frac{2}{|c|}$ must be a simple eigenvalue of the Hessian $H(c)$.
\end{remark}

\section{Algebras of Cartan's isoparametric cubics}
\label{sec:isop}

In this section we consider the principal model example  which clarifies some  algebraic ingredients of the construction of nonclassical solutions. We establish the corresponding Peirce decomposition and derive the finiteness of the eiconal algebras in Proposition~\ref{pro:division}.

\subsection{Cartan-M\"unzner equations}
 The only known so far example of nonclassical solution in the lowest dimension $n=5$ (see Theorem~\ref{th12}) is based on the cubic form  $u_5$ that naturally appears in the context of isoparametric hypersurfaces in the Euclidean spheres.

According to \'{E}.~Cartan, a hypersurface of the unit sphere $S^{n-1}\subset \R{n}$ is called \textit{isoparametric} if it has constant principal curvatures. Isoparametric hypersurfaces have been shown to be useful in various areas of mathematics, see \cite{CecilRyan} for the modern account of the isoparametric theory. Cartan himself classified all isoparametric hypersurfaces with $g=1,2,3$ constant principal curvatures and also established that all these are homogeneous and algebraic. The celebrated  result of M\"unzner \cite{Mun1} states that $g\in \{1,2,3,4,6\}$, all five possibilities are realized, and each isopaprametric hypersurface with $g$ distinct principal curvatures is obtained as a level set of a homogeneous degree $g$ polynomial on the unit sphere.

The case $g=3$ is very distinguished in many way. In \cite{Cartan39MZ} Cartan proved that any isoparametric hypersurface $M^{3d}$ with $g=3$ distinct principal curvatures, each principal curvature must have the same multiplicity, and the possible multiplicities are
$
d=1,2,4,8
$
corresponding the dimensions of classical division algebras $\F_d$.
More precisely, $M^{3d}$ is a tube of constant radius over a standard Veronese embedding of a projective plane into the standard sphere over the division algebra $\F_{d}$. Equivalently, $M^{3d}$ is a locus of a cubic form $u(x)$ in $S^{3d+1}\subset\R{3d+2}$ with $u$ satisfying the  Cartan-M\"unzner system \index{Cartan-M\"unzner system}
\begin{equation}\label{Muntzer0}
|\nabla u(x)|^2=9|x|^{4},
\end{equation}
\begin{equation}\label{Muntzer1}
\Delta u(x)=0.
\end{equation}
Cartan classified all cubic solutions of (\ref{Muntzer0})--(\ref{Muntzer1}) and showed that the corresponding defining cubic polynomials are given explicitly by\index{Cartan's isoparametric cubic}
\begin{equation}\label{CartanFormula0}
\begin{split}
u &=x_{1}^3+\half{3}{2}x_{1}(|z_1|^2+|z_2|^2-2|z_3|^2-2x_{2}^2)
+\half{3\sqrt{3}}{2}x_{2}(|z_2|^2-|z_1|^2)
\\&+{3\sqrt{3}}\re (z_1z_2)z_3,
\end{split}
\end{equation}
where $x=(x_1,x_2,z_1,z_2,z_3)$ and $z_i\in \R{d}\cong \F_d$ and $d\in\{1,2,4,8\}$. We refer to \eqref{CartanFormula0} as to a \textit{Cartan isoparametric cubic}.
Each Cartan isoparametric cubic also satisfies a determinantal representation like \eqref{u5} above, where the latter determinant should be properly understood in an appropriate sense. More precisely,  $u$ is the generic determinant in the Jordan algebra of $3\times 3$-Hermitian matrices with entries in the division algebra $\F_d$.

\subsection{Algebras attached to \eqref{Muntzer0}}
Below, we apply the definitions given in section~\ref{sec:algebrascubic} to Cartan isoparametric cubics. Somewhat different approach using the Freu\-den\-thal-Springer construction was suggested  in \cite{Tk14}.

Our starting point is an arbitrary cubic homogeneous polynomial solution $u(x)$ of the Cartan-M\"unzner equation \eqref{Muntzer0} {alone}. By abusing of terminology, we call $u$ an \textit{eiconal cubic}. The harmonic eiconal cubics are exactly the Cartan isoparametric cubics. Using (\ref{grad}), we introduce the commutative algebra structure $V(u)$ on $V=\R{n}$ equipped with the standard Euclidean inner product. Then $u(x)=\frac16\scal{x^2}{x}$ and \eqref{Muntzer0} becomes
$$
\scal{x^2}{x^2}=36\scal{x}{x}^2,
$$
The exact value of the constant factor $36$ is not essential and may be chosen arbitrarily (positive) by a suitable scaling of the inner product.

Now we want to consider an arbitrary algebra satisfying the above identity. This motivates the following definition.

\begin{definition}
A commutative, maybe nonassociative,  algebra with a positive definite associating form $\scal{}{}$ satisfying
\begin{equation}\label{Cartandef}
\scal{x^2}{x^2}=\scal{x}{x}^2
\end{equation}
is called an \textit{eiconal} algebra.
\end{definition}

If $V$ is an arbitrary eiconal algebra then the cubic form $u(x)=\frac16\scal{x}{x^2}$ satisfies (scaled) eiconal equation \eqref{Muntzer0}. This translates the study of \eqref{Muntzer0} into a purely algebraic context.

Note also that the sense of \eqref{Cartandef} becomes more clear if one introduces the norm $N(x)=\scal{x}{x}$. Then  \eqref{Cartandef} takes form of the composition algebra identity
\begin{equation}\label{compose}
N(x^2)=N(x)^2.
\end{equation}
Note, however, that \eqref{compose} does not imply that $N(xy)=N(x)N(y)$

Our goal is the Peirce decomposition of $V$. To this end we need the standard linearization technique which is an important tool in nonassociative algebra, see \cite{McCrbook}. More precisely,  we  linearize \eqref{Cartandef} at $x$ \textit{in the direction} $y$ to get
$$
4\scal{xy}{x^2}=4|x|^2\scal{x}{y}.
$$
Since the inner product is associating, we have
$$
\scal{xy}{x^2}=\scal{y}{x^2x}=\scal{y}{x^3}.
$$
(Note that  by virtue of the commutativity of $V$ the third power in $V$ is well defined: $x^3=x^2x=xx^2$.) Therefore
$\scal{x^3-\scal{x}{x}x}{y}=0$ holds for all $y\in V,$
implying  by the nondegeneracy of the inner product that $x^3-\scal{x}{x}x=0$ for all $x$. Conversely, if the latter identity holds, one easily gets \eqref{Cartandef}. This proves

\begin{proposition}
An arbitrary commutative  algebra with a positive definite associating form $\scal{}{}$ is eiconal if and only if
\begin{equation}\label{Cartan2}
x^3=\scal{x}{x}x, \qquad \forall x\in V.
\end{equation}
\end{proposition}

Next note that the harmonicity condition \eqref{Muntzer1} is equivalent (taking into account \eqref{hessL}) to the trace free condition
\begin{equation}\label{tracefree}
\trace L_x=0, \qquad \forall x\in V.
\end{equation}
A further linearization of \eqref{Cartan2} in direction $y\in V$ yields
$$
x^2y+2x(xy)=\scal{x}{x}y+2\scal{x}{y}x,
$$
and eliminating $y$, we get
\begin{equation}\label{Cartan3}
\phantom{\frac{1}{2}}L_{x^2}+2L_x^2=\scal{x}{x}+2x\otimes x.
\end{equation}

\begin{remark}
The latter identity implies that eiconal algebras are `nearly Jordan'. Indeed, recall any Jordan algebra satisfies \eqref{Jordandef}. On the other hand, we have from \eqref{Cartan2} and \eqref{Cartan3}
\begin{equation*}
\begin{split}
[L_{x^2},L_x^2]&=2[x\otimes x, L_x^2]=2(x\otimes x^3-x^3\otimes x)=0,
\end{split}
\end{equation*}
i.e. $L_{x^2}$ commutes  with $L_{x}^2$.  In fact, it follows from \cite{Tk14} that any eiconal algebra has a natural structure of the trace free  subspace in a  rank 3 Jordan algebra.
\end{remark}

\subsection{The Peirce decomposition and fusion laws}
It follows from the definition of an eiconal algebra $V$ that $x^2\ne0$ for any $x\ne0$, i.e. $V$ is a nonzero algebra. Therefore Proposition~\ref{pro:main} ensures that there are nonzero idempotents in $V$. Let us denote by $\Idm(V)$ the set of all nonzero idempotents. Note also that by \eqref{Cartandef}
$
\text{$|c|=1$  for any $c\in \Idm(V)$}.
$

The multiplication operator $L_c$ is self-adjoint  with respect to the inner product $\scal{}{}$, hence $V$ decomposes into orthogonal sum of $L_c$-invariant subspaces. Let us determine the spectrum of $L_c$. To this end, note  that $c^2=c$, hence $c$ is an eigenvector of $L_c$ with eigenvalue $1$, thus, $\R{}c$ is an invariant subspace of $L_c$. Therefore, the orthogonal complement $c^\bot$ is an invariant subspace of $L_c$ too. Applying \eqref{Cartan3} we obtain
\begin{equation}\label{Lccc}
2L_c^2+L_c-1=2c\otimes c=0 \quad \text{on } c^\bot,
\end{equation}
hence $\mathrm{spectrum}(L_c|_{c^\bot})\subset \{-1,\half12\}.$
Note that it follows from the above inclusion that the eigenvalue $1$ has multiplicity one, i.e. any idempotent in $V$ is \textit{primitive}.

\begin{remark}
We point out that the presence of the eigenvalue $\frac12$ is crucial for constructing the nonclassical solutions and closely related to the generacy condition discussed in sections~\ref{sec:trely} and \ref{sec:generic} above, cf. with  Lemma~3.2 in \cite{NV10a}.
\end{remark}
Let $V_{\lambda}(c)$ be the $\lambda$-eigenspace of $L_c$. Then the Peirce decomposition of $V$ is
$$
V=\R{}c\oplus V_{-1}(c)\oplus V_{\frac12}(c).
$$

In order  to extract the multiplication table (the so-called fusion laws) between the eigenspaces $V_{\lambda}(c)$, we linearize further (\ref{Cartan3}). This   yields
\begin{equation}\label{Cartan4}
L_{yx}+(L_{x}L_y+L_yL_x)=\scal{x}{y}+x\widehat{\otimes} y.
\end{equation}
Applying \eqref{Cartan4} to an arbitrary element $z\in V$ yields the full linearization
\begin{equation}\label{Cartan40}
x(yz)+y(zx)+z(xy)=\scal{x}{y}z+\scal{y}{z}x+\scal{z}{x}y.
\end{equation}
Specializing  $z=c\in \Idm(V)$ in \eqref{Cartan40} and setting $x,y\in c^\bot$ yields
\begin{equation*}\label{Cartan6}
(cx)y+x(cy)+c(xy)=\scal{x}{y}c.
\end{equation*}
Taking the scalar product with $z$ in the latter identity and assuming that
$x\in V_{\lambda_1}(c)$, $y\in V_{\lambda_1}(c)$ and $z\in V_{\lambda_3}(c)$, where $\lambda_i\in\{\half12, -1\}$, we obtain
\begin{equation}\label{Cartan7}
\scal{x_1x_2}{x_3}(\lambda_1+\lambda_2+\lambda_3)=0.
\end{equation}
As a corollary we have
$$
V_{\lambda_1}(c)V_{\lambda_2}(c)\,\bot\, V_{\lambda_3}(c)\quad \text{whenever $\lambda_1+\lambda_2+\lambda_3\ne0$.}
$$
For example, setting $\lambda_1=\lambda_2=-1$  immediately implies that $V_{-1}(c)V_{-1}(c)$ is perpendicular to both $V_{-1}(c)$ and $V_{\frac12}(c)$, hence
\begin{equation}\label{Vcminus}
V_{-1}(c)V_{-1}(c)\subset \R{}c.
\end{equation}
Similarly, for $\lambda_1=-1$ and $\lambda_2=\half12$ one has $V_{-1}(c)V_{\frac12}(c)\subset \R{}c\oplus V_{\frac12}(c)$. On the other hand, since eigenspaces $V_{\lambda}(c)$ are perpendicular for distinct $\lambda$ we have
$$
\scal{V_{-1}(c)V_{\frac12}(c)}{c}=\scal{V_{-1}(c)}{V_{\frac12}(c)c}= \scal{V_{-1}(c)}{V_{\frac12}(c)}=0,
$$
implying  that $V_{-1}(c)V_{\frac12}(c)\subset V_{\frac12}(c)$. Arguing similarly for $\lambda_1=\lambda_2=\half12$, one arrives at the fusion (multiplication) laws shown in   Table~\ref{tab:eic}.

\begin{table}[h]
\begin{tabular}{c|cc}
         \vphantom{$\int\limits^{}$} & $V_{-1}$ & $V_{\frac12}$
         \smallskip  \\\hline
$V_{-1}$  \vphantom{$\int\limits^{}$}&  $\R{}c$ & $V_{\frac12}$ \\
\smallskip

$V_{\frac12}$ \vphantom{$\int\limits^{}$}& $V_{\frac12}$ & $\R{}c\oplus V_{-1}$ \smallskip\\
 \end{tabular}
 \caption{Fusion laws of an eiconal algebra}\label{tab:eic}
\end{table}

Recall that a linear map $A:X\times Y\to Y$ such that $A(x,\cdot):Y\to Y$ is self-adjoint for all $x\in X$ and $A^2(x,\cdot)=\scal{x}{x}\, \mathrm{id}_Y$ is called a symmetric Clifford system, cf.  \cite{CecilRyan}, \cite{Ovsienko2016}. It is well-known that in this case $\dim Y$ is even and $\dim X\le 1+\rho(\frac12\dim Y)$, and  the Hurwitz-Radon function $\rho$ is defined by
\begin{equation}\label{radon2}
\rho(m)=8a+2^b, \qquad \text{if} \;\,\,m=2^{4a+b}\cdot \mathrm{odd} , \;\; 0\leq b\le 3.
\end{equation}

\begin{proposition}
\label{pro:division}
If $\dim V\ge2$ then
\begin{equation}\label{radon1}
\dim V_{-1}(c)-1\le \rho(\half12\dim V_{\frac12}).
\end{equation}
If additionally $V$ satisfies \eqref{tracefree} then the possible dimensions of the Peirce subspace $V_{-1}(c)$ coincides with these of classical division algebras. In particular, there are only finitely many isomorphy classes of harmonic eiconal algebras in dimensions $\dim V\in \{5,8,14,26\}$.
\end{proposition}

\begin{proof}
First we show that $V_{-1}(c)$ is nontrivial. Indeed, assume by contradiction that $V_{-1}(c)=\{0\}$. Then $V=\R{}c\oplus V_{\frac12}(c)$, hence $\dim V_{\frac12}(c)=n-1\ge1$. From Table~\ref{tab:eic},  $V_{\frac12}(c)V_{\frac12}(c)\subset \R{}c$. Hence for any $x\in V_{\frac12}(c)$ we have
$$
x^2=\scal{x^2}{c}c=\scal{x}{xc}c=\half12\scal{x}{x}c.
$$
Therefore $x^3=\half14\scal{x}{x}x\ne \scal{x}{x}x$, a contradiction with (\ref{Cartan2}) follows. Thus $\dim V_{-1}(c)\ge 1$.  Next, note that for any  $x\in V_{-1}(c)$ Table~\ref{tab:eic} yields that $L_x$ is an endomorphism of $V_{\frac12}(c)$. Since by (\ref{Vcminus})
$$
x^2=\scal{x^2}{c}c=\scal{x}{xc}c=-\scal{x}{x}c,
$$
we find from  (\ref{Cartan3})
\begin{equation}\label{Cartan10}
L_x^2=\half34\scal{x}{x} \text{ on }V_{\frac12}(c).
\end{equation}
Thus, $\frac{2}{\sqrt{3}}L_x$ is a symmetric Clifford system. This  imposes the dimensional obstruction \eqref{radon1}. If additionally $V$ satisfies \eqref{tracefree} then
$$
0=\trace L_c=1+\frac12 \dim V_{\frac12}(c)-\dim V_{-1}(c),
$$
thus $\dim V_{\frac12}(c)=2m$, where $m=\dim V_{-1}(c)-1$, implying by \eqref{radon1} $\rho(m)\ge m$. It is well known and also easily follows from  \eqref{radon2} that the latter inequality holds only if $m=1,2,4,8$, which implies that $\dim V=3m+2$.

\end{proof}


\section{Concluding remarks and  open questions}\label{sec:last}
Minimal cones constitutes an important subclass of singular minimal hypersurfaces in the Euclidean space $\R{n}$ for $n\ge4 $ plying a crucial role in study of both the local and global structures of general (regular) minimal hypersurfaces.
Many known so far examples of minimal hypercones are algebraic varieties coming essentially from two classical algebraic structures: the Jordan and the Clifford algebras.

An important class is  minimal cones defined by a degree three homogeneous polynomials, i.e. cubic minimal cones. These were considered by W.Y.~Hsiang  in the late 1960s in \cite{Hsiang67}. On the other hand, all known examples of  cubic forms satisfying the hyperbolicity criteria \ref{itemhess1} or \ref{itemhess2} (including the examples discussed above and the determinant form $u_9$ in Remark ~\ref{rem23}) are \textit{exceptional} Hsiang cubic minimal cones (or radial eigencubics in the sense of chapter 6 in \cite{NTVbook}). It is known that there are only finitely many (congruence classes of) such cubics in some distinguished dimensions $5\le n\le 72$, see Table~1 on p.~158 in  \cite{NTVbook}. From analytical point of veiw, any exceptional Hsiang cubic form can be characterized as a cubic polynomial solution of the  following Hessian trace equations:
\begin{equation}\label{trace2}
\begin{split}
\trace D^2u(x)&=0,\\
\trace (D^2u(x))^2&=C_1 |x|^2,\\
\trace (D^2u(x))^3&=C_2 u(x),
\end{split}
\end{equation}
Natural questions arise: Does any exceptional Hsiang cubic produce a nonclassical solution? Do there exist special cubics which are not Hsiang eigencubics? In general, it would be interesting to clarify the connections between minimal cones and construction of  nonregular solutions.


\bibliographystyle{plain}
\def\cprime{$'$}

\end{document}